\documentclass[11pt]{article}
\usepackage{graphicx}

\usepackage{amsfonts, amssymb, url, amsthm, bbm, amsmath, hyperref, tikz-cd, comment, stmaryrd, mathrsfs}

\usepackage{xcolor}

\newcommand{\bZ}{\mathbb{Z}}
\newcommand{\bQ}{\mathbb{Q}}
\newcommand{\Zp}{\bZ_p}
\newcommand{\Qp}{\bQ_p}

\newcommand{\cL}{\mathcal{L}}
\newcommand{\IZind}{\mathrm{ind}_{IZ}^G\>}

\newcommand{\Fq}{\mathbb{F}_q}

\newcommand{\id}{\mathrm{id}}
\newcommand{\GL}{\mathrm{GL}}

\newcommand{\SymE}[1]{\underline{\mathrm{Sym}}^{#1}E^2}

\newcommand{\SymF}[1]{\mathrm{Sym}^{#1}\Fq^2}

\newcommand{\logL}{\log_{\cL}}

\newcommand{\latticeL}[1]{\tilde{\Theta}(#1, \cL)}
\newcommand{\ind}{\mathrm{ind}\>}

\newcommand{\br}[1]{\overline{#1}}
\newcommand{\tU}{\tilde{U}_p}
\newcommand{\tW}{\tilde{W}_p}

\theoremstyle{plain}
\newtheorem{theorem}{Theorem}[section]
\newtheorem{lemma}[theorem]{Lemma}

\newtheorem{Proposition}[theorem]{Proposition}
\newtheorem{remark}[theorem]{Remark}

\title{Reduction mod $p$ of semi-stable representations of some super-Breuil weights}

\author{Anand Chitrao and Eknath Ghate}

\begin{document}

\maketitle

\begin{abstract}
    We determine the mod $p$ reductions of the semi-stable representations $V_{k, \cL}$ of weight $k \in [p + 5, 2p]\cup[2p + 6, 3p + 1]$ and $v_p(\cL) < 1-k/2$ for primes $p \geq 5$. 
    In particular, this shows that the techniques introduced in \cite{CG23} involving the $p$-adic and mod $p$ local Langlands correspondences can be used to compute the reduction of $V_{k, \cL}$ outside the range $k \in [3, p + 1]$.
    Moreover, this shows that the bound on $v_p(\cL)$ given by Bergdall-Levin-Liu \cite{BLL23} can be improved, at least for weights $k \in [2p + 6, 3p + 1]$.
\end{abstract}

\section{Introduction}
    Fix a prime $p \geq 5$. Let $E$ be a finite extension of $\Qp$ containing $\sqrt{p}$. For $k \geq 3$, let $V_{k, \cL}$ be the two-dimensional irreducible semi-stable representation with Hodge-Tate weights $(0, k - 1)$ and $\cL$-invariant $\cL \in E$. In our earlier paper \cite{CG23}, we computed the reduction mod $p$ of $V_{k, \cL}$ for $k \in [3, p + 1]$ and for arbitrary $\cL$, extending the work of \cite{BM02} and \cite{GP18} for $k \in [3, p - 1]$. To achieve this, we computed the reduction mod $p$ of the Banach space associated to $V_{k, \cL}$ by the $p$-adic local Langlands correspondence (see \cite{Bre04, Bre10, Col10}).
    As was noted in \cite{CG23}, our method works, in theory, for arbitrary $k \geq 3$. It was put into practice by computing the reductions completely for the weights $k = p$ and $p + 1$. 

    In \cite{BLL23}, Bergdall-Levin-Liu used the theory of Breuil and Kisin modules to compute the reductions of $V_{k, \cL}$ for arbitrary $k \geq 4$ but for $\cL$ satisfying $v_p(\cL) < 2 - k/2 - v_p((k - 2)!)$ (even for $p = 3$). In Section $1.3$ of their paper, the authors mention that a computation using R. Pollack's code shows that the conclusion of their result holds for $p = 3$, $k = 6$, and $v_p(\cL) = -2$. Note that $2 - k/2 - v_p((k - 2)!) = -2$, hinting at the fact that their bound could possibly be improved. Indeed, in \cite[Remark 1.2]{BLL23}, the authors wonder to what extent their bound is optimal. In \cite{CG23}, we showed that their bound is indeed optimal for $k \in [3, p + 1]$. For $k \in [3, p - 1]$, this was already known by \cite{BM02} and \cite{GP18}.

    The goal of this paper is twofold. On the one hand, we wish to recover some of the results in \cite{BLL23} using our methods. On the other hand, we wish to see whether the bound given in \cite{BLL23} can be improved.
    
    In order to do this, we compute the reduction of $V_{k, \cL}$ for $k$ belonging to $[p + 5, 2p] \cup [2p + 6, 3p + 1]$ and $\cL$ satisfying $v_p(\cL) < -r/2$, where $r = k - 2$. When $k \in [p + 5, 2p]$, we recover the result in \cite{BLL23}.
    Since for $k \in [2p + 6, 3p + 1]$, the expression $2 - k/2 - v_p((k - 2)!)$ equals $-1 - r/2$, which is strictly smaller than $-r/2$, our result improves the bound in \cite{BLL23} for $k \in [2p + 6, 3p + 1]$. 

    We use the notation in \cite{CG23}.
    We write most of this paper assuming that $p \leq r \leq p^2 - p - 1$ and that $v_p(\cL)$ is slightly more general than $v_p(\cL) < -r/2$. 
    We only specialize to $r \in [p + 3, 2p - 2] \cup [2p + 4, 3p - 1]$ and $v_p(\cL) < -r/2$ in the last section, where we prove the following result.
    \begin{theorem}\label{Intro theorem}
        Let $p \geq 5$ and $p + 3 \leq r \leq 2p - 2$ or $2p + 4 \leq r \leq 3p - 1$. Then, for any $\cL$ with $v_p(\cL) < -r/2$, we have $\overline{V}_{k, \cL} \simeq \ind \omega_2^{r + 1}$.
    \end{theorem}  

    To prove Theorem~\ref{Intro theorem}, we recall that by \cite[(9)]{CG23}, there is a surjection
    \[
        \IZind \SymF{r} \twoheadrightarrow \br{\latticeL{k}}.
    \]
    The canonical $IZ$-equivariant filtration on $\SymF{r}$ induces a $G$-equivariant filtration on $\br{\latticeL{k}}$. We name the sub-quotients of this filtration on $\br{\latticeL{k}}$ in such a way that the sub-quotient $\IZind a^id^{r - i}$ of $\IZind \SymF{r}$ surjects onto the sub-quotient $F_{2i, 2i + 1}$ of $\br{\latticeL{k}}$ for $0 \leq i \leq r$.

    The proof of the theorem proceeds by eliminating all but one of the sub-quotients $F_{2i, 2i+1}$. The remaining sub-quotient is the one contributing towards $\br{\latticeL{k}}$. We then use the Iwahori theoretic version of the mod $p$ local Langlands correspondence given in \cite[Theorem 2.2]{Chi25} to see that this sub-quotient must be supersingular.

    We explain the strategy of the proof in detail now. First, using \cite[Lemma 6.1]{CG23}, we may eliminate $F_{2i, 2i+1}$ for $i > r/2$. Next, consider the following diagram of the $F_{2i, 2i + 1}$ for $0 \leq i \leq \lfloor r/2 \rfloor$
    {\small
    \[
    \begin{array}{ccccccc}
        p^{r/2 - \lceil r/2 \rceil}z^{\lceil r/2 \rceil} 
        & \ldots \qquad & p^{r/2 - (r - 2)}z^{r - 2} & p^{r/2 - (r - 1)}z^{r - 1} & p^{r/2 - r}z^r \vspace{.1cm} \\ \vspace{.2cm}
        F_{2\lfloor r/2 \rfloor, 2\lfloor r/2 \rfloor + 1} 
        & \ldots \qquad & F_{4, 5} & F_{2, 3} & F_{0, 1} \\
        z^{\lfloor r/2 \rfloor} 
        & \ldots \qquad & z^2 & z & z^0.
    \end{array}
    \]
    }
    \!\!\!The bottom row displays the usual generator of $F_{2i, 2i+1}$, which is the image of the function $\llbracket\id, X^{i}Y^{r - i}\rrbracket$ under the surjection above (although we must multiply by the common function $\mathbbm{1}_{p\Zp}$), and the top row displays $-\beta = -\left(\begin{smallmatrix}0 & 1 \\ p & 0\end{smallmatrix}\right)$ applied to the usual generator (again with $\mathbbm{1}_{p\Zp}$ everywhere).
    
    In our earlier paper \cite{CG23}, the sub-quotient which contributed in the $\nu < 1 - r/2$ case was $F_{0, 1}$. This time around, however, $F_{0, 1}$ will be eliminated. This requires a new idea, which is described in Section~\ref{Killing shallow sub-quotients}. It was a surprise to us that $F_{0, 1}$ can be eliminated without using the poly$\cdot$log functions (so called in \cite{Gha26}) so prevalent in \cite{CG23}! It turns out that for $i \geq 0$, the sub-quotient $F_{2i, 2i+1}$ can be eliminated if $r \geq (i + 1)(p + 1) - 1$.

    The shallowest sub-quotient that is not eliminated by this easy method is the contributing one. Therefore, we skip that sub-quotient and move on to eliminating the remaining ones, which we henceforth refer to as the ``deeper'' sub-quotients. In Section~\ref{Killing deep sub-quotients}, we prove a master congruence in the form of Proposition~\ref{Final congruence}, which will be used multiple times to eliminate all the deeper sub-quotients. One can view this congruence as an extension of the $\nu < 1 + r/2 - n$ case of \cite[Proposition $9.6$]{CG23} to higher values of $r$. Stirling numbers of the second kind also make an appearance.

    In Section~\ref{Improving the BLL bound}, we eliminate the deeper sub-quotients by proving three mutually exclusive and exhaustive theorems depending on the power of $z$ in the top row in the diagram above.
    Theorems~\ref{The good method} and \ref{The bad method} use the congruence in Proposition~\ref{Final congruence} for a single value of $n$ satisfying $v_p([n]_{b + 1}) = 0$ and $1$, respectively, where $b = \lfloor n/p \rfloor$. We refer to the arguments given in these theorems as the good method and the bad method, respectively. Theorem~\ref{The ugly method} requires invoking Proposition~\ref{Final congruence} twice, once for $n$ satisfying $v_p([n]_{b + 1}) = 1$ and the other for $n$ satisfying $v_p([n]_{b + 1}) = 0$. We shall refer to this as the ugly method\footnote{With apologies to the 1966 cult classic.}.

    Our main theorem then follows from a simple application of the Iwahori mod $p$ LLC.

    \section{Preliminaries}
        First, we state some combinatorial lemmas.
        We begin with the mod $p^2$ version of Lucas' theorem, proved by Rowland.
        \begin{lemma}[{\cite[Theorem 4]{Row22}}]\label{Lucas mod p^2}
            If $0 \leq s \leq r \leq p - 1$, $a \geq 0$, and $b \geq 0$, then
            \[
                {pa + r \choose pb + s} \equiv {a \choose b}{r \choose s}\left(1 + pa(H_{r} - H_{r - s}) + pb(H_{r - s} - H_{s})\right) \mod p^2.
            \]
        \end{lemma}
        \noindent This recovers classical Lucas after going mod $p$. We will also use the fact that $H_{p - 1} \equiv 0 \mod p$ for $p \geq 3$ since $H_{p - 1}$ consists of the sum of a complete set of non-zero residues mod $p$.
        
        We also need Stirling numbers of the second kind, which we recall are defined by (see \cite{SP00})
        \[
            {t \brace s} = \frac{1}{s!}\sum_{j = 0}^{s}(-1)^{j}{s \choose j}(s - j)^t \text{ for integers $t, s \geq 0$}.
        \]
        It can be shown by applying the operator $\left(x\frac{d}{dx}\right)^{t}$ to the binomial expansion of $(1 + x)^s$ that ${t \brace s} = 0$ if $s > t$. We adopt the convention that ${t \brace s} = 0$ if $s < 0$.
        
        While manipulating the congruences obtained using Proposition~\ref{Final congruence}, we make use of a version of Lucas for Stirling numbers of the second kind proved by S\'anchez-Peregrino.
        \begin{lemma}[{\cite[Corollary 2.2]{SP00}}]\label{Lucas for Stirling}
            For integers $i, x, y \geq 0$,
            \[
                {y + p^i \brace x} \equiv {y + 1 \brace x} + \sum_{j = 1}^{i}{y \brace x - p^j} \mod p.
            \]
        \end{lemma}

    Finally, we prove the following existence lemma for the coefficients $\lambda_i$ for certain poly$\cdot$log functions which we use later.
    \begin{lemma}\label{New coefficient identities}
    Let $b p \leq n \leq (b + 1) p - 1$ for an integer $b \geq 0$. If $I = \{0, 1, \ldots, n, (b + 1)p\}$ and $z_i = i$ for $i \in I$, then there are $\lambda_i \in \Zp$ such that
    \begin{itemize}
        \item $\sum_{i \in I}\lambda_i z_i^j = 0$ for all $0 \leq j \leq n$,
        \item $\sum_{i \equiv a \!\!\!\mod p} \lambda_iz_i^j \equiv 0 \mod p^2$ for a fixed $0 \leq a \leq p - 1$ and all $0 \leq j \leq n$,
        \item $\lambda_i \equiv (-1)^{b - i/p}{b + 1 \choose i/p} \>\equiv (-1)^{bp - i}{(b + 1)p \choose i} \mod p$ if $p \mid i$,
        \item $\lambda_i \equiv 0 \> \equiv (-1)^{bp - i}{(b + 1)p \choose i}\mod p$ if $p \nmid i$.
    \end{itemize}
    Moreover, we may choose $\lambda_{(b + 1)p} = -1$ in $\Zp$.
\end{lemma}
\begin{proof}
    Set $\lambda_{(b + 1)p} = -1$. We move $\lambda_{(b + 1)p}((b + 1)p)^j$ in the first bullet point to the RHS and get
    \[
        \sum_{i = 0}^{n}\lambda_i i^j = ((b + 1)p)^j
    \]
    for each $0 \leq j \leq n$.
    Using the determinant formula for Vandermonde matrices, we see that this system of equations has a unique solution over $\Qp$. We use Cramer's rule to show that 
    the solution is integral and satisfies the three remaining bullet points.

    Let $A$ be the coefficient matrix for the system of equations displayed above. For $0 \leq i \leq n$, let $A_i$ be the matrix $A$ where the column with index $i$ is replaced by the column vector on the right.

    Using Vandermonde's determinant formula, we see that
    \begin{eqnarray}\label{Value of lambda_i}
        \lambda_i = \frac{\det A_i}{\det A} & = & \prod_{\substack{i' = 0 \\ i' \neq i}}^n\frac{(b + 1)p - i'}{i - i'} 
         = \frac{[(b + 1)p]_{n + 1}}{i!(n - i)!(-1)^{n - i}((b + 1)p-i)} \nonumber \\
        & = & (-1)^{n - i}{(b + 1)p \choose n + 1}\frac{(n + 1)}{(b + 1)p - i}{n \choose i} \nonumber \\
        & = & (-1)^{n - i}\frac{(b + 1)p}{(b + 1)p - i}{(b + 1)p - 1 \choose n}{n \choose i}
    \end{eqnarray}
    for $0 \leq i \leq n$.
    It suffices to prove the second bullet point for all $a \neq 0$ since the second bullet point for $a = 0$ follows from this and the first bullet point. Thus, it suffices to prove that for $1 \leq a \leq p - 1$, we have
    \[
        \sum_{i \equiv a \!\!\!\!\!\mod p}(-1)^i\frac{i^j}{(b + 1)p - i}{n \choose i} \equiv 0 \mod p.
    \]
    Write $n = bp + \epsilon$ and $i = kp + a$ for some $0 \leq \epsilon \leq p - 1$ and $k$ is a non-negative integer. By Lucas' theorem, we have
    \[
        {n \choose i} \equiv {b \choose k}{\epsilon \choose a} \mod p.
    \]
    If $\epsilon < a$, then this is congruent to $0$ modulo $p$ for all $k$. If $\epsilon \geq a$, then
    \begin{eqnarray*}
        \sum_{i \equiv a \!\!\!\!\!\mod p}(-1)^i\frac{i^j}{(b + 1)p - i}{n \choose i} & \equiv & \sum_{k = 0}^{b}-(-1)^{kp + a} (kp + a)^{j - 1}{b \choose k}{\epsilon \choose a} \\
        & \equiv & (-1)^{a + 1}a^{j - 1}{\epsilon \choose a}\sum_{k = 0}^{b}(-1)^k{b \choose k} = 0 \mod p.
    \end{eqnarray*}
    
    We prove the third bullet point. So assume $p \mid i$. Using ${m \choose n}{n \choose i} = {m \choose i}{m - i \choose n - i}$ and ${p - 1 \choose j} \equiv (-1)^j \mod p$, we see that
    \begin{eqnarray*}
        \lambda_i & = & (-1)^{n - i}\frac{(b + 1)p}{(b + 1)p - i}{(b + 1)p - 1 \choose i}{(b + 1)p - 1 - i \choose n - i} \\
        & = & (-1)^{n - i}{(b + 1)p \choose i}{(b - i/p)p + (p - 1) \choose (b - i/p)p + \epsilon} \equiv (-1)^{bp - i + 2\epsilon}{(b + 1)p \choose i} \mod p.
    \end{eqnarray*}

    The fourth bullet point is immediate from \eqref{Value of lambda_i}.
    \end{proof}

\section{Killing shallow sub-quotients}\label{Killing shallow sub-quotients}
    In this section, we first show that for $r \geq p$ and any $\cL$, the sub-quotient $F_{0, 1}$ of $\latticeL{k}$ dies modulo $\pi$. 
    
    Consider the canonical map with dense image
    \begin{eqnarray}\label{Iwahori uniformization}
        \frac{\IZind \SymE{r}}{(\tU - 1, \tW + 1)} \to \latticeL{k}.
    \end{eqnarray}
    Let $f(X, Y) = X^{p - 1}Y^{r - p + 1} - Y^{r}$ be a polynomial in $\SymE{r}$. This polynomial makes sense if $r \geq p - 1$. Using the formulae written on the bottom of page 7 in \cite{CG23}, we see that
    \begin{eqnarray*}
        (\tU\tW)\llbracket \id, f(X, Y)\rrbracket & = & \sum_{\lambda \in I_1} \left \llbracket \begin{pmatrix} \lambda & 1 \\ 1 & 0\end{pmatrix}, \begin{pmatrix}0 & 1 \\ 1 & -\lambda \end{pmatrix} \cdot f(X, Y) \right \rrbracket\\
        & = & \sum_{\lambda \in I_1} \left\llbracket \begin{pmatrix} \lambda & 1 \\ 1 & 0\end{pmatrix}, \left(Y^{p - 1}(X - \lambda Y)^{r - p + 1} - (X - \lambda Y)^{r}\right)\right\rrbracket \\
        & = & \sum_{\lambda \in I_1}\left \llbracket \begin{pmatrix} \lambda & 1 \\ 1 & 0 \end{pmatrix}, (-\lambda)^{r - p + 1}Y^r - (-\lambda)^r Y^r + g(X, Y) \right \rrbracket,
    \end{eqnarray*}
    where $g(X, Y)$ is a homogeneous polynomial in $X$ and $Y$ of degree $r$ with no pure power of $Y$. If $r \geq p$, then $r - p + 1 > 0$ and $(-\lambda)^{r - p + 1} = (-\lambda)^{r}$ for all $\lambda \in I_1$. (If $r = p - 1$, then this is not true for $\lambda = 0$.) This computation therefore shows that the image of $\tU\tW \llbracket \id, f(X, Y) \rrbracket$ under \eqref{Iwahori uniformization} is integral. Furthermore, this image projects to $0$ in $F_{0, 1}$.
    
    Next, using the fact that $\tU = 1$ and $\tW = -1$ on the left side of \eqref{Iwahori uniformization}, we see that $-\llbracket \id, f(X, Y)\rrbracket = \tU\tW \llbracket \id, f(X, Y) \rrbracket$. Therefore $-\llbracket \id, f(X, Y) \rrbracket$ projects to $0$ in $F_{0, 1}$. However, $-\llbracket \id, f(X, Y)\rrbracket = \llbracket \id, Y^r - X^{p - 1}Y^{r - p + 1} \rrbracket$ projects to a generator of $F_{0, 1}$. So, $F_{0, 1} = 0$.

    More generally for $i \geq 1$, we prove that the sub-quotient $F_{2(i - 1), 2(i - 1) + 1}$ vanishes if $r \geq i(p + 1) - 1$. Note that the function $\llbracket \id, X^{i - 1}Y^{r - i + 1} \rrbracket$ projects to a generator of $F_{2(i - 1), 2(i - 1) + 1}$. Define 
    \[
        f_i(X, Y) = \frac{Y^{r - i(p + 1) + 1}}{X} (-\theta)^{i},
    \]
    where $\theta = X^pY - XY^p$ is the Dickson polynomial.
    Using the definition of $\tU$ and $\tW$, we see that
    \begin{eqnarray*}
         (\tU\tW)\llbracket \id, f_i(X, Y) \rrbracket & = & \sum_{\lambda \in I_1} \left\llbracket \begin{pmatrix} \lambda & 1 \\ 1 & 0\end{pmatrix}, \frac{(X - \lambda Y)^{r - i(p + 1) + 1}}{Y} (\theta)^{i}\right\rrbracket \mod p,
    \end{eqnarray*}
    since $\GL_2(\Zp)$ acts on $\theta$ mod $p$ by $\det$ mod $p$. Since the largest power of $X$ in this polynomial is $i$, we see that it projects to $0$ in $F_{2(i - 1), 2(i - 1) + 1}$. On the other hand, as before, the LHS is $-\llbracket \id, f_i(X, Y) \rrbracket$, which projects to a generator of $F_{2(i - 1), 2(i - 1) + 1}$. Therefore $F_{2(i - 1), 2(i - 1) + 1} = 0$.
    
\section{Towards killing deep sub-quotients}\label{Killing deep sub-quotients}
In the spirit of \cite{CG23}, we first prove a congruence converting poly$\cdot$log functions into locally polynomial functions on $\Qp$ modulo $\pi\latticeL{k}$. 
Throughout this paper, $n$ is an integer in $[\lfloor r/2 \rfloor + 1, r]$
Recall that $b = \lfloor n/p \rfloor$. Define, $[n]_{m} = n (n - 1) \cdots (n - m + 1)$.

\begin{lemma}\label{Telescoping lemma}
    Let $5 \leq p \leq r \leq p^2 - p - 1$, $r/2 + b + 1 \leq n \leq r$, 
    and $z_0 \in \Zp$. Let $v_p(\cL) \leq r/2 - n$ and fix $x \in \bQ$ such that 
    $x + v_p(\cL) = r/2 - n - v_p([n]_{b + 1})$. Then for 
    \begin{enumerate}
        \item $g(z) = p^x(z - z_0)^n\logL(z - z_0)\mathbbm{1}_{\Zp}$
        \item $g(z) = p^x z^{r - n} (1 - zz_0)^n\logL(1 - zz_0)\mathbbm{1}_{p\Zp},$
    \end{enumerate}
    we have $g(z) \equiv g_2(z) \mod \pi\latticeL{k}$.
\end{lemma}
\begin{remark}\label{Bound on x} 
    Note that $x + v_p(\cL) = r/2 - n - v_p([n]_{b + 1})$ and $v_p(\cL) \leq r/2 - n$ imply that $x \geq - v_p([n]_{b + 1}) \geq -1$.
    Also, the condition $r/2 + b + 1 \leq n$ implies that $n - v_p([n]_{b + 1}) > r/2$. This can be seen separately for $n < p$ and $n \geq p$.
    Also, $r/2 + b + 1 \leq r$ implies that $r \geq 2(b + 1)$.
\end{remark}
\begin{proof}
    This generalizes \cite[Lemma 8.4]{CG23}. Since the proof is just some general algebra and \cite[Lemma 8.3]{CG23}, we only prove the analog of \cite[Lemma 8.3]{CG23} here.

    So we show that for $h \geq 3$, $0 \leq a \leq p^{h- 1} - 1$, $0 \leq \alpha \leq p - 1$ and $0 \leq j \leq n - 1$, we have
    \begin{eqnarray}\label{Congruence in telescoping lemma}
        \frac{g^{(j)}(a + \alpha p^{h - 1})}{j!} \equiv \frac{g^{(j)}(a)}{j!} + \alpha p^{h - 1}\frac{g^{(j + 1)}(a)}{j!} + \cdots + \frac{(\alpha p^{h - 1})^{n - 1 - j}}{j!(n - 1 - j)!}g^{(n - 1)}(a) \mod (p^{h - 1})^{r/2 - j}\pi.
    \end{eqnarray}

    First  assume that $g(z) = p^x(z - z_0)^n\logL(z - z_0)\mathbbm{1}_{\Zp}$.
    \begin{itemize}
        \item $v_p(a - z_0) < h - 1$: Write the expansion
        \begin{eqnarray}\label{Taylor expansion close to z_0}
            \frac{g^{(j)}(a + \alpha p^{h - 1})}{j!} & = & \frac{g^{(j)}(a)}{j!} + \cdots + \frac{(\alpha p^{h - 1})^{n - 1 - j}}{j!(n - 1 - j)!}g^{(n - 1)}(a) \\
            & & + \frac{(\alpha p^{h - 1})^{n - j}}{j!(n - j)!}n!p^{x}[\logL(a - z_0) + H_n] \nonumber \\
            & & + \sum_{l \geq 1}\frac{(\alpha p^{h - 1})^{n - j + l}}{j!(n - j + l)!}n! p^{x} \frac{(-1)^{l - 1}}{(a - z_0)^l}(l - 1)!.\nonumber
        \end{eqnarray}

        The valuation of the term in the second line on the RHS is $\geq (h - 1)(n - j) + r/2 - n - v_p([n]_{b + 1})$ due to:
        \begin{itemize}
            \item $\frac{n!}{j!(n - j)!} \in \Zp$ for all $j$,
            \item $v_p(\logL(a - z_0)) \geq \min\{v_p(\cL), 1\} = v_p(\cL)$,
            \item $v_p(H_n) \geq - \lfloor \log_p(n) \rfloor \geq -1 \geq r/2 - n \geq v_p(\cL).$
        \end{itemize}
        This is strictly greater than $(h - 1)(r/2 - j)$ for $h \geq 3$ because $n - v_p([n]_{b + 1}) > r/2$.

        Next, the valuation of the $l$-th summand in the last sum on the RHS is $\geq - v_p([n]_{b + 1}) + (h - 1)(n - j) + l - \log_p(n - j + l)$ by looking at the display just below \cite[(20)]{CG23}, which we copy here for ease of checking:
        \[
            p^{x}(\alpha p^{h - 1})^{n - j}\frac{n!}{j!(n - j)!}\left(\frac{\alpha p^{h - 1}}{a - z_0}\right)^{l}(-1)^{l - 1}\frac{1}{l{n - j + l \choose l}}.
        \]
        So for $h \geq 3$ and $l \geq 1$, we have to check
        \[
            - v_p([n]_{b + 1}) + (h - 1)(n - j) + l - \log_p(n - j + l) > (h - 1)(r/2 - j).
        \]
        This is equivalent to
        \[
            p^{- v_p([n]_{b + 1}) + 2(n - r/2) + l} > n + l
        \]
        for $l \geq 1$, which is the special case $h = 3$ and $j = 0$.
        \begin{itemize}
            \item Base case: LHS $= p^{2(n - r/2) - v_p([n]_{b + 1}) + 1} \geq p^2 > n + 1 =$ RHS. This is true for $p \geq 2$.
            \item Derivative: $\mathrm{LHS}' = p^{2(n - r/2) - v_p([n]_{b + 1}) + l}\mathrm{\ln}\> p > (n + 1)\mathrm{ln}\>p > 1 = \mathrm{RHS}'$. This is true for $p \geq 2$.
        \end{itemize}
        So the graph of the LHS lies above the graph of the RHS for $l \geq 1$.
        Therefore, we have shown that the second and the third lines on the RHS of \eqref{Taylor expansion close to z_0} die modulo $(p^{h - 1})^{r/2 - j}\pi$.

    \item $v_p(a - z_0) \geq h - 1$: Using the derivative formula:
    \[
        g^{(l)}(z) = p^x\frac{n!}{(n - l)!}((z - z_0)^{n - l}\logL(z - z_0) + (H_n - H_{n - l})(z - z_0)^{n - l})\mathbbm{1}_{\Zp},
    \]
    with $l = j, j + 1, \ldots, n - 1$, we can write
    \[
        (\alpha p^{h - 1})^{l - j}\frac{g^{(l)}(z)}{j!(l - j)!} = (\alpha p^{h - 1})^{l - j}p^x{l \choose j}{n \choose l}((z - z_0)^{n - l}\logL(z - z_0) + (H_n - H_{n - l})(z - z_0)^{n - l})\mathbbm{1}_{\Zp}.
    \]
    Using this computation with $z = a + \alpha p^{h - 1}$ on the LHS and $z = a$ on the RHS of \eqref{Congruence in telescoping lemma}, we see that each term there has valuation $\geq (h - 1)(n - j) + r/2 - n - v_p([n]_{b + 1})$. This is clearly greater than $(h - 1)(r/2 - j)$ for $h \geq 3$ as $n - v_p([n]_{b + 1})> r/2$.

    So in this case, \eqref{Congruence in telescoping lemma} turns out to be $0 \equiv 0 \mod (p^{h - 1})^{r/2 - j}\pi$ exactly as in the analogous case in \cite[Lemma 8.3]{CG23}.
    \end{itemize}

    Next, assume that $g(z) = p^xz^{r - n}(1 - zz_0)^n\logL(1 - zz_0)\mathbbm{1}_{p\Zp}$. For this function, we do not need to write down cases.

    Since this function is identically $0$ outside $p\Zp$ and therefore the same holds for all its derivatives, we prove \eqref{Congruence in telescoping lemma} only for $p \mid a$. Exactly as in \cite{CG23}, we use the fact that $\log(1 + T)$ has a Taylor expansion at $0$ for $p \mid T$ and we write the Taylor expansion of $g^{(j)}(a + \alpha p^{h - 1})$ as:
    \begin{eqnarray}\label{Taylor expansion far from z_0}
        \frac{g^{(j)}(a + \alpha p^{h - 1})}{j!} & = & \frac{g^{(j)}(a)}{j!} + \cdots + \frac{(\alpha p^{h - 1})^{n - 1 - j}}{j!(n - 1 - j)!}g^{(n - 1)}(a) \\
        && + \sum_{l \geq 1}\frac{(\alpha p^{h - 1})^{n - 1 - j + l}}{j!(n - 1 - j + l)!}g^{(n - 1 + l)}(a), \nonumber
    \end{eqnarray}
    where the derivatives for $j \geq 0$ are given by the generalized Leibnitz's rule as in \cite[(22)]{CG23}:
    \begin{eqnarray*}
        g^{(j)}(z) & = & p^x\sum_{m = 0}^{j}{j \choose m}\left.\begin{cases} \dfrac{(r - n)!z^{r - n - j + m}}{(r - n - j + m)!} & \!\!\!\! \text{ if } m \geq j - r + n \\ 0 & \!\!\!\!\text{ if } m < j - r + n\end{cases}\right\} \> \cdot \>  (-z_0)^m  \\
        && \left.\begin{cases}\dfrac{n!}{(n - m)!}\left[(1 - zz_0)^{n - m}\logL(1 - zz_0) + (H_n - H_{n - m})(1 - zz_0)^{n - m}\right] & \!\!\!\! \text{ if } m \leq n \\ \dfrac{n!(-1)^{m - n - 1}(m - n - 1)!}{(1 - zz_0)^{m - n}} & \!\!\!\! \text{ if } m > n\end{cases}\right\}\mathbbm{1}_{p\Zp}(z). \nonumber
    \end{eqnarray*}
The only change is that we have written the exact formula for $*$. Using this formula, we see that the valuation of $g^{(j)}(a)$ is greater than or equal to $- v_p([n]_{b + 1})$. The bound on $x$ is the reason for
this, and everything else is integral. Here we need the fact that $\frac{n!}{(n - m)!}(H_n - H_{n - m}) \in \Zp$.

So, the $l$-th term in the second line on the RHS of \eqref{Taylor expansion far from z_0} has valuation $\geq - v_p([n]_{b + 1}) + (h - 1)(n - 1 - j + l) - v_p((n - 1 + l)!) \geq (h - 1)(n - 1 - j + l) - \frac{n - 1 + l}{p - 1} - v_p([n]_{b + 1})$. Now we have to check
\[
    (h - 1)(n - 1 - j + l) - \frac{n - 1 + l}{p - 1} - v_p([n]_{b + 1}) > (h - 1)(r/2 - j)
\]
for $h \geq 3$, $l \geq 1$, $0 \leq j \leq n - 1$. Equivalently, we must prove
\[
    (h - 1)(n - 1 - r/2 + l) - v_p([n]_{b + 1}) > \frac{n - 1 + l}{p - 1},
\]
for $h \geq 3$, $l \geq 1$. Since the term that is multiplied with $(h - 1)$ is positive and since the RHS is independent of $h$, it is enough to prove
\[
    2(n - 1 - r/2 + l) - v_p([n]_{b + 1}) > \frac{n - 1 + l}{p - 1},
\]
for $l \geq 1$. Noting $n - r/2 \geq b + 1$, it is enough to check
\[
    2(b + l) - v_p([n]_{b + 1}) > \frac{n - 1 + l}{p - 1}.
\]
Rearranging, it is enough to check
\[
    \left(2 - \frac{1}{p - 1}\right)l + 2b - 1 > \frac{n - 1}{p - 1}
\]
for $l \geq 1$. It is enough to check the base case $l = 1$, which is
\[
    2b + 1 > \frac{n}{p - 1}.
\]
This is easily checked.
Therefore we have proved \eqref{Congruence in telescoping lemma} for the second function. \qedhere
\end{proof}

\begin{lemma}\label{qp-zp part is 0}
    Let $5 \leq p \leq r \leq p^2 - p - 1$ and $r/2 + b + 1 \leq n \leq r$. Let $x + v_p(\cL) = r/2 - n - v_p([n]_{b + 1})$ and let $v_p(\cL) \leq r/2 - n$. 
    Let $I$, $z_i$ and $\lambda_i$ be as in Lemma \ref{New coefficient identities} and let
    \[
        f(z) = \sum_{i \in I}p^x \lambda_iz^{r - n}(1 - zz_i)^n\logL(1 - zz_i)\mathbbm{1}_{p\Zp}.
    \]
    Then, we have $f(z) \equiv 0 \mod \pi \latticeL{k}$.
\end{lemma}
\begin{proof}
    By Lemma \ref{Telescoping lemma} (2), we see that $f(z) \equiv f_2(z) \mod \pi \latticeL{k}$.

    Recall
    \[
        f_2(z) = \sum_{a = 0}^{p^2 - 1}\sum_{j = 0}^{n - 1}\frac{f^{(j)}(a)}{j!}(z - a)^j\mathbbm{1}_{a + p^2\Zp}.
    \]

    Since $f(z)$ has $\mathbbm{1}_{p\Zp}$ as a factor, so do its derivatives and therefore the $p\nmid a$ summands in the display above can be dropped. For the $p \mid a$ summands, write $a = \alpha p$ for $0 \leq \alpha \leq p - 1$:
    \[
        f_2(z) = \sum_{\alpha = 0}^{p - 1}\sum_{j = 0}^{n - 1}\frac{f^{(j)}(\alpha p)}{j!}(z - \alpha p)^j\mathbbm{1}_{\alpha p + p^2\Zp}.
    \]

    Now the derivative of $f$ is written in the proof of Lemma~\ref{Telescoping lemma}. 
    Since $j \leq n - 1$, it simplifies to:
    \begin{eqnarray*}
        f^{(j)}(z) & = & \sum_{i \in I}p^x\lambda_i\sum_{m = 0}^{j}{j \choose m}\left.\begin{cases} \dfrac{(r - n)!z^{r - n - j + m}}{(r - n - j + m)!} & \!\!\!\! \text{ if } m \geq j - r + n \\ 0 & \!\!\!\!\text{ if } m < j - r + n\end{cases}\right\} \> \cdot \>  (-z_i)^m  \\
        && \left[\dfrac{n!}{(n - m)!}(1 - zz_i)^{n - m}\logL(1 - zz_i) + \dfrac{n!}{(n - m)!}(H_n - H_{n - m})(1 - zz_i)^{n - m}\right] \mathbbm{1}_{p\Zp}(z). \nonumber
    \end{eqnarray*}
    We have distributed $\frac{n!}{(n - m)!}$ over the terms in the square brackets to make the coefficients of powers of $z$ in the second term integral.

    Substitute $z = \alpha p$ in $f^{(j)}(z)$. Then expanding
    \[
        (1 - \alpha p z_i)^{n - m}\logL(1 - \alpha p z_i) = \sum_{l \geq 1}c_l(\alpha p z_i)^l
    \]
    with $c_l \in \bQ$ satisfying $v_p(c_l) \geq - \lfloor \log_p(l) \rfloor$, as well as the second term in the square brackets, and collecting like powers of $z = \alpha p$, we get a Taylor expansion of the expression in the square brackets. Next, we use the identities $\sum_{i \in I}\lambda_i z_i^j = 0$ for $0 \leq j \leq n$ to kill the terms in this Taylor expansion of degree $\leq n - m$. We have to shift from $n$ to $n - m$ because of the $(-z_i)^m$ factor at the end of the first line.

    To show that $f_2(z) \equiv 0 \mod \pi\latticeL{k}$ using \cite[Lemma 6.3]{CG23}, we have to show that
    \begin{eqnarray}\label{Wishing that derivatives have large valuation}
        v_p(f^{(j)}(\alpha p)) - v_p(j!) > r/2 - j.
    \end{eqnarray}

    By Remark~\ref{Bound on x}, the valuation of $f^{(j)}(\alpha p)$ is
    \[
        \geq \min_{l}\{x + r - n - j + m + l - \lfloor\log_p(l)\rfloor\} \geq \min_{l}\{- 1 + r - n - j + m + l - \log_p(l)\},
    \]
    where the minimum is taken over $l \geq n - m + 1$.

    It suffices to check
    \[
        -1 + r - n - j + m + l - \log_p(l) - v_p(j!) > r/2 - j
    \]
    for $l \geq n - m + 1$. It's enough to check
    \[
       -1 + r - n + m + l - \log_p(l) - v_p(j!) > r/2.
    \]
    Doing some algebra, the question becomes
    \[
        p^{-1 + r/2 - n + m + l - v_p(j!)} \stackrel{?}{>} l.
    \]
    \begin{itemize}
        \item Base case: LHS $= p^{r/2 - v_p(j!)} \geq p^{r/2 - v_p(r!)} > r + 1 \geq $ RHS for $p \geq 5$.
        \item Derivative: LHS$' = p^{-1 + r/2 - n + m + l -v_p(j!)}\ln(p) > (r + 1)\ln(p) > 1 =$ RHS$'$.
    \end{itemize}

    This computation proves \eqref{Wishing that derivatives have large valuation}. So $f_2(z) \equiv 0 \mod \pi\latticeL{k}$.
\end{proof}

Using the lemmas above, we can now prove a master congruence.

\begin{Proposition}\label{Final congruence}
    Let $5 \leq p \leq r \leq p^2 - p - 1$. Let $n, b$ be non-negative integers such that $bp \leq n \leq (b + 1)p - 1$, and $r/2 + b + 1 \leq n \leq r$. For $v_p(\cL) < r/2 - n$, set
        \[
            g(z) = p^x\left[\sum_{i \in I} \lambda_i (z - i)^n\logL(z - i)\right],
        \]
    where $x \in \bQ$ with $x + v_p(\cL) = r/2 - n - v_p([n]_{b + 1})$ and $I$, $\lambda_i \in \Zp$ as in Lemma~\ref{New coefficient identities}.
    Then, we have
        \begin{eqnarray*}
            g(z) & \equiv & \sum_{a = 1}^{n - bp}\sum_{j = \lceil r/2 \rceil}^{n - 1}{n \choose j}{\epsilon \choose a}p^{x + n - j}\frac{(-1)^{- a - j + b - 1}}{a}{(b + 1)p \choose n + 1}(n + 1)\cL b!{n - j \brace b}(z - a)^j\mathbbm{1}_{a + p\Zp} \\
            && \quad + \sum_{j = \lceil r/2 \rceil - 1}^{n - 1}{n \choose j}p^{x + n - j}(-1)^{n - j}\cL *_j z^j\mathbbm{1}_{p\Zp} \mod \pi \latticeL{k},
        \end{eqnarray*}
        where
    \begin{eqnarray*}
        *_j & = & {(b + 1)p - 1 \choose n}(-1)^{n}\Bigg[(-1)^{b + 1}{n - j \brace b + 1}(b + 1)! + pH_{\epsilon}(-1)^{b + 1}{n - j + 1 \brace b + 1}(b + 1)!  \\
        && \qquad \qquad \qquad \qquad \qquad \qquad -(-1)^{b + 1}(b + 1)^{n - j} - pH_{\epsilon}(-1)^{b + 1}(b + 1)^{n - j + 1}\Bigg] - (b + 1)^{n - j}.
    \end{eqnarray*}
    In particular, 
    \[
        *_j \equiv -{n - j \brace b + 1}(b + 1)! - pH_{\epsilon}{n - j \brace b}(b + 1)! \mod p^2
    \]
    and therefore,
    \begin{eqnarray*}
        *_j & = & \begin{cases}
            0 \mod p & \text{ if } j \geq n - b \\
            -(b + 1)! \mod p & \text{ if } j = n - b - 1,
        \end{cases} \\
        *_j & = & 
        \begin{cases}
            0 \mod p^2 & \text{ if } n - b + 1 \leq j \leq n - 1 \\
            -p H_{\epsilon}(b + 1)! \mod p^2 & \text{ if } j = n - b \\
            -(b + 1)! - pH_{\epsilon}(b + 1)!{b + 1 \choose 2} \mod p^2 & \text{ if } j = n - b - 1.
        \end{cases}
    \end{eqnarray*}
\end{Proposition}
\begin{proof}
    Before beginning the proof, we remark that $v_p(\cL) \lneq r/2 - n$ implies $x \gneq - v_p([n]_{b + 1})$ (cf. Remark~\ref{Bound on x}). Also, $0 \leq b \leq p - 2$.
    
    Write $g(z) = g(z)\mathbbm{1}_{\Zp} + g(z)\mathbbm{1}_{\Qp \setminus \Zp}$. By Lemma~\ref{Telescoping lemma} and Lemma~\ref{qp-zp part is 0}, we see that $g(z) \equiv g_2(z) \mod \pi \latticeL{k}$.

    We show that $g_2(z) - g_1(z) \equiv 0 \mod \pi \latticeL{k}$ as in \cite[Proposition 9.6]{CG23}. The coefficient of $(z - a - \alpha p)^j \mathbbm{1}_{a + \alpha p + p^2 \Zp}$ in $g_2(z) - g_1(z)$ for $0 \leq a, \alpha \leq p - 1$ and $0 \leq j \leq n - 1$ from \cite[(49)]{CG23} is
    \begin{eqnarray}\label{jth term in g2 - g1}
        && {n \choose j}\sum_{i \in I}p^x \lambda_i\left[(a + \alpha p - i)^{n - j}\logL(a + \alpha p - i) - (a - i)^{n - j}\logL(a - i) \right. \\
        && \left. - {n - j \choose 1}(\alpha p)(a - i)^{n - j - 1}\logL(a - i) - \cdots - {n - j \choose n - j - 1}(\alpha p)^{n - j - 1}(a - i)\logL(a - i)\right]. \nonumber
    \end{eqnarray}
    We consider two cases:
            \begin{enumerate}
            \item Sum of the $i \not \equiv a \mod p$ terms.
            
            Let $i \not \equiv a \mod p$. The $i$-th summand in \eqref{jth term in g2 - g1} is
            \[
                p^x \lambda_i[(a - i + \alpha p)^{n - j}\logL(1 + (a - i)^{-1}\alpha p) + (\alpha p)^{n - j}\logL(a - i)].
            \]
            The valuation of the second term in the brackets is at least $x + n - j + 1$. This is greater than $n - j + 1 - v_p([n]_{b + 1}) > r/2 - j$. So we may ignore the second term using \cite[Lemma $6.3$]{CG23}.
            Expand the first term as 
            \[
                (a - i)^{n - j}(1 + (a - i)^{-1}\alpha p)^{n - j}\logL(1 + (a - i)^{-1}\alpha p) = (a - i)^{n - j}\sum_{l \geq 1}c_l(a - i)^{-l}(\alpha p)^l,
            \]
            where $v_p(c_l) \geq - \lfloor \log_p(l) \rfloor$ as was used many times in \cite{CG23} (these $c_l$ have slightly different indices than the $c_l$ in \cite[(50)]{CG23}). As in \cite[(50)]{CG23}, we show that the terms with $l \geq n - j$ vanish mod $p^{r/2 - j}\pi$. Indeed, the valuation of the $l$-th term (including $p^x \lambda_i)$ is $\geq x + l - \lfloor \log_p(l) \rfloor > - v_p([n]_{b + 1}) + l - \lfloor \log_p(l) \rfloor$. The last expression is clearly greater than or equal to $r/2 - j$ for $l = n - j$.

            For $l \geq n - j + 1$, it is enough to check
            \[
                -v_p([n]_{b + 1}) + l - \log_p(l) \geq r/2 - j.
            \]
            
            This is equivalent to
            \[
                p^{- r/2 + j - v_p([n]_{b + 1}) + l} \geq l.
            \]
            \begin{itemize}
                \item Base case: LHS $= p^{n - r/2 + 1 - v_p([n]_{b + 1})} \geq p^{b + 2 - v_p([n]_{b + 1})} \geq r + 1 \geq n - j + 1 =$ RHS. This is true for $p \geq 2$.
                \item Derivative: LHS$' = p^{ - r/2 + j - v_p([n]_{b + 1}) + l}\ln(p) \geq (r + 1)\ln(p) > 1 =$ RHS$'$. This is true for $p \geq 2$.
            \end{itemize}
            
            So the $i$-th summand can be written as
            \begin{eqnarray}\label{i neq a mod p in g2 - g1}
                p^x\lambda_i[(a - i)^{n - j - 1}\alpha p + c_{2}(a - i)^{n - j - 2}(\alpha p)^2 + \cdots + c_{n - j - 1}(a - i)(\alpha p)^{n - j - 1}] \mod p^{r/2 - j}\pi.
            \end{eqnarray}

            Summing \eqref{i neq a mod p in g2 - g1} over the $i \not \equiv a \mod p$ terms and using $\sum_{i \in I}\lambda_i (z - i)^{l} = 0$ for $0 \leq l \leq n$ (a consequence of $\sum_{i \in I}\lambda_i i^l = 0$), the sum becomes
            \[
                \sum_{i \equiv a \!\!\!\!\mod p} -p^x\lambda_i[(a - i)^{n - j - 1}\alpha p + c_{2}(a - i)^{n - j - 2}(\alpha p)^2 + \cdots + c_{n - j - 1}(a - i)(\alpha p)^{n - j - 1}] \mod p^{r/2 - j}\pi.
            \]
            Each term above vanishes since it has valuation $\geq x + n - j -\lfloor \log_p(n) \rfloor > n - v_p([n]_{b + 1}) - j -\lfloor \log_p(n) \rfloor \geq r/2 - j$ as $n \geq r/2 + b + 1$.
            
            \item Sum of the $i \equiv a \mod p$ terms.
            
            Let $i \equiv a \mod p$. We prove that we can replace all log terms in \eqref{jth term in g2 - g1} with $\cL$. This is similar to \cite[Claim (26)]{CG23}. For instance, we prove
            \begin{eqnarray*}
                p^x\lambda_i(a - i)^{n - j}\logL(a - i) & \equiv & p^x\lambda_i(a - i)^{n - j}\cL \mod p^{r/2 - j}\pi.
            \end{eqnarray*}
            This is obvious if $a - i = 0$ as $j < n$. When $v_p(a - i) = 1$, write $a - i = pu$ for some unit $u$. Therefore
            \[
                p^x\lambda_i(a - i)^{n - j}\logL(a - i) = p^x\lambda_i(a - i)^{n - j}\cL + p^x\lambda_i(a - i)^{n - j}\logL(u).
            \]
            The valuation of the second term on the RHS of the equation above is $\geq x + n - j + 1 > n + 1 - v_p([n]_{b + 1}) - j > r/2 - j$. The case $2 \leq v_p(a - i) < \infty$ cannot happen as $-(b + 1)p \leq a - i \leq p - 1$ and $0 \leq b \leq p - 2$. This proves the claim.

            Therefore, using the binomial theorem, we see that each of the $i \equiv a \mod p$ summands in \eqref{jth term in g2 - g1} is congruent to 
            \[
                p^x \lambda_i (\alpha p)^{n - j}\cL
            \]
            modulo $p^{r/2 - j}\pi$. Using the fact that $\sum_{i \equiv a \!\!\mod p}\lambda_i \equiv 0 \mod p^2$ (by Lemma~\ref{New coefficient identities}), we see that the sum of the $i \equiv a \mod p$ terms is $0$ mod $p^{r/2 - j}\pi$.
            \end{enumerate}
            Therefore $g_2(z) - g_1(z) \equiv 0 \mod \pi\latticeL{k}$. Note that so far we have just used $v_p(\lambda_i) \geq 0$ and not the stronger inequality $v_p(\lambda_i) \geq 1$ for $i \not \equiv 0 \mod p$, which will be used shortly below.
    
    Thus $g(z) \equiv g_1(z) \mod \pi\latticeL{k}$, where
    \begin{eqnarray}\label{g1}
        g_1(z) = \sum_{a = 0}^{p - 1}\sum_{j = 0}^{n - 1}\frac{g^{(j)}(a)}{j!}(z - a)^j \mathbbm{1}_{a + p\Zp}.
    \end{eqnarray}
    We have
    \[
        \frac{g^{(j)}(a)}{j!} = \sum_{i \in I}{n \choose j}p^x\lambda_i(a - i)^{n - j}\logL(a - i).
    \]
    To compute this sum, we consider two cases:
    
    \begin{itemize}
        \item $a \neq 0$.
            
            The $i \not \equiv a \!\!\mod p$ terms are $0$ modulo $\pi$ using \cite[Lemma 6.3]{CG23} because $x > -1$ and $p \mid \logL(a - i)$.

            So assume $i \equiv a \mod p$. Using $v_p(\lambda_i) \geq 1$ from Lemma~\ref{New coefficient identities}, note that $x + v_p(\lambda_i) + n - j + v_p(\cL) \geq r/2 - j + 1 - v_p([n]_{b + 1}) \geq r/2 - j$. We use \cite[Lemma $6.3$]{CG23} to see that the $j < r/2$ terms are $0$. In the $j \geq r/2$ terms, we replace all logs by $\cL$ modulo $p^{r/2 - j}\pi$ as was done for $g_2(z) - g_1(z)$ and write the $i$-th summand as
            \[
                {n \choose j}p^x\lambda_i(a - i)^{n - j}\cL.
            \]
            Substituting the value of $\lambda_i$ obtained in \eqref{Value of lambda_i} the display above becomes
            \[
                {n \choose j}p^x(-1)^{n - i}{(b + 1)p \choose n + 1}\frac{(n + 1)}{(b + 1)p - i}{n \choose i}(a - i)^{n - j}\cL.
            \]
            Substituting $i = a + lp$, we see that the sum of the $i \equiv a \mod p$ terms is
            \[
                {n \choose j}p^{x + n - j}(-1)^{n-a}{(b + 1)p \choose n + 1}(n + 1)\cL\sum(-1)^l\frac{1}{(b + 1)p - a - lp}{n \choose a + lp}(-l)^{n - j},
            \]
            where the sum ranges over $0 \leq l \leq b - 1$ if the $0$-th $p$-adic digit $\epsilon$ of $n$ satisfies $\epsilon < a$ or over $0 \leq l \leq b$ otherwise.

            The valuation of the expression outside the sum is $\geq x + n - j + 1 + v_p(\cL) = r/2 - j + 1 - v_p([n]_{b + 1})$. Therefore to evaluate the expression above modulo $p^{r/2 - j}\pi$, we have to evaluate the sum modulo $\pi$.

            If $\epsilon < a$, then this sum is $0$ by Lucas' theorem. However, if $\epsilon \geq a$, then the sum is
            \[
                {\epsilon \choose a}\sum_{l = 0}^{b}(-1)^l\frac{(-l)^{n - j}}{-a}{b \choose l} \mod p.
            \]
            This sum is
            \[
                {\epsilon \choose a} \frac{(-1)^{n - j}}{(-a)} b! (-1)^b {n - j \brace b},
            \]
            where ${n - j \brace b}$ is a Stirling number of the second kind. Therefore if $p \leq n \leq p^2 - p - 1$, the $a \neq 0$ summands in $g_1(z)$ are
            \[
                \sum_{a = 1}^{n - bp}\sum_{j = \lceil r/2 \rceil}^{n - 1}{n \choose j}{\epsilon \choose a}p^{x + n - j}\frac{(-1)^{- a - j + b - 1}}{a}{(b + 1)p \choose n + 1}(n + 1)\cL b!{n - j \brace b}(z - a)^j\mathbbm{1}_{a + p\Zp} \mod \pi\latticeL{k}.
            \]
        \item $a = 0$.

            The $i \not \equiv 0 \!\!\mod p$ terms in $\frac{g^{(j)}(a)}{j!}$ are $0$ modulo $\pi$ using \cite[Lemma 6.3]{CG23} because $x > -1$ and $p \mid \logL(a - i)$.

            The $i \equiv 0 \!\!\mod p$ terms are: 
            \[
                {n \choose j}p^x \lambda_i (-i)^{n - j}\logL(-i) \equiv {n \choose j}p^x\lambda_i(-i)^{n - j}\cL \mod p^{r/2 - j}\pi.
            \]
            We have used the fact that $(b + 1)p < p^2$ here. Substituting the value of $\lambda_i$, this equals
            \begin{eqnarray*}
                {n \choose j}p^x (-1)^{n - i}{(b + 1)p \choose n + 1}\frac{(n + 1)}{(b + 1)p - i}{n \choose i} (-i)^{n - j}\cL & \text{ if } i \neq (b + 1)p \\
                {n \choose j}p^x(-1)(-i)^{n - j}\cL & \text{ if } i = (b + 1)p.
            \end{eqnarray*}
            If $b = 0$, then the sum of the $i \equiv 0 \mod p$ terms is
            \[
                -{n \choose j}p^{x + n - j}(-1)^{n - j}\cL.
            \]
            If $b = 1$, the sum of the $i \equiv 0 \mod p$ terms is
            \[
                {n \choose j}p^{x + n - j}(-1)^{n - j}\cL\left((-1)^{n - p}{2p \choose n + 1}\frac{(n + 1)}{p}{n \choose p} - 2^{n - j}\right).
            \]
            For a general $b \leq p - 2$, the sum of the $i \equiv 0 \mod p$ terms is
            \begin{eqnarray*}\label{sum of i eqiv 0 mod p terms in g1}
                && {n \choose j}p^{x + n - j}(-1)^{n - j}{(b + 1)p \choose n + 1}\frac{(n + 1)}{p} \cL \left(\sum_{\substack{i = 0 \\ p \mid i}}^{bp}(-1)^{n - i}\frac{(i/p)^{n - j}}{b + 1 - i/p}{n \choose i}\right) \\
                && \quad - {n \choose j}p^{x + n - j}(-1)^{n - j}(b + 1)^{n - j}\cL.
            \end{eqnarray*}
            
            Note that if $j < r/2 - 1$, then the display above is $0$ modulo $\pi$ since its valuation is $\geq x + n - j + v_p(\cL) = r/2 - j - v_p([n]_{b + 1}) > 0$. So terms in \eqref{g1} with $a = 0$, $j < r/2 - 1$ are $0$ mod $\pi\latticeL{k}$ by \cite[Lemma 6.3]{CG23}. However, terms with $j \geq r/2 - 1$ possibly survive. 

            Since the terms outside the parentheses in the first line have valuation $\geq x + n - j + v_p(\cL) = r/2 - j - v_p([n]_{b + 1})$, we only need to evaluate the sum in the parentheses modulo $p^2$. Writing $n = bp + \epsilon$, $i = kp$, and using Lucas' theorem 
            for the values of binomial coefficients mod $p^2$ (Lemma~\ref{Lucas mod p^2}), the sum becomes
            \begin{eqnarray*}
                && \equiv \frac{(-1)^n}{b + 1}\left[\sum_{k = 0}^{b}(-1)^k\frac{k^{n - j}(b + 1)}{b + 1 - k}{b \choose k}(1 + pkH_{\epsilon})\right] \mod p^2 \\
                && = \frac{(-1)^n}{b + 1}\left[\sum_{k = 0}^{b + 1}(-1)^kk^{n - j}{b + 1 \choose k}(1 + pkH_{\epsilon}) - (-1)^{b + 1}(b + 1)^{n - j}(1 + p(b + 1)H_{\epsilon})\right] \\
                && = \frac{(-1)^n}{b + 1}\Bigg[(-1)^{b + 1}{n - j \brace b + 1}(b + 1)! + pH_{\epsilon}(-1)^{b + 1}{n - j + 1 \brace b + 1}(b + 1)!  \\
                && \qquad \qquad \qquad \qquad -(-1)^{b + 1}(b + 1)^{n - j} - pH_{\epsilon}(-1)^{b + 1}(b + 1)^{n - j + 1}\Bigg].
            \end{eqnarray*}

            Putting
            \begin{eqnarray*}
                *_j & := & {(b + 1)p - 1 \choose n}(-1)^n\Bigg[(-1)^{b + 1}{n - j \brace b + 1}(b + 1)! + pH_{\epsilon}(-1)^{b + 1}{n - j + 1 \brace b + 1}(b + 1)!  \\
                && \qquad \qquad \qquad \qquad -(-1)^{b + 1}(b + 1)^{n - j} - pH_{\epsilon}(-1)^{b + 1}(b + 1)^{n - j + 1}\Bigg] - (b + 1)^{n - j},
            \end{eqnarray*}
            we see that the $a = 0$ summands in $g_1(z)$ are
            \[
                \sum_{j = \lceil r/2 \rceil - 1}^{n - 1}{n \choose j}p^{x + n - j}(-1)^{n - j}\cL *_j z^j\mathbbm{1}_{p\Zp} \mod \pi \latticeL{k}.
            \]
    \end{itemize}
    Therefore, we finally get

    \begin{eqnarray*}
        g(z) & \equiv & \sum_{a = 1}^{n - bp}\sum_{j = \lceil r/2 \rceil}^{n - 1}{n \choose j}{\epsilon \choose a}p^{x + n - j}\frac{(-1)^{- a - j + b - 1}}{a}{(b + 1)p \choose n + 1}(n + 1)\cL b!{n - j \brace b}(z - a)^j\mathbbm{1}_{a + p\Zp} \\
        && + \sum_{j = \lceil r/2 \rceil - 1}^{n - 1}{n \choose j}p^{x + n - j}(-1)^{n - j}\cL *_j z^j\mathbbm{1}_{p\Zp} \mod \pi \latticeL{k}.
    \end{eqnarray*}

    Next, we compute the value of $*_j$ modulo $p^2$. We write $n = bp + \epsilon$. 
    Using the mod $p^2$ version of Lucas and $H_{p - 1} \equiv 0 \mod p$, the binomial coefficient times the sign becomes
    \begin{eqnarray*}
        {p - 1 \choose \epsilon}(1 + bp(H_{p - 1} - H_{\epsilon}))(-1)^{b + \epsilon} & = & \frac{(1 - p)(2 - p) \cdots (\epsilon - p)}{\epsilon!}(1 - bpH_{\epsilon})(-1)^b \\
            & \equiv & (1 - pH_{\epsilon})(1 - bpH_{\epsilon})(-1)^b \mod p^2 \\
            & \equiv & (1 - (b + 1)pH_{\epsilon})(-1)^b \mod p^2.
    \end{eqnarray*}
    Substituting this in the equation above, we get
    \begin{eqnarray*}
        *_j & \equiv & - {n - j \brace b + 1}(b + 1)! + pH_{\epsilon}\left[-{n - j + 1 \brace b + 1}(b + 1)! + {n - j \brace b + 1}(b + 1)(b + 1)!\right] \\
        & \equiv & -{n - j \brace b + 1}(b + 1)! - pH_{\epsilon}{n - j \brace b}(b + 1)! \mod p^2.
    \end{eqnarray*}
    In the last step, we have used the recurrence relation for Stirling's numbers of the second kind given, e.g., in \cite[Section $1$]{SP00}. The final two claims regarding the values of $*_j$ modulo $p$ and $p^2$ can now be proved using properties of Stirling's numbers.
\end{proof}

In the next section, we eliminate most sub-quotients $F_{2i, 2i + 1}$ of $\br{\latticeL{k}}$ when $cp + c + 2 \leq r \leq (c + 1)p - 1$ for $c = 1, 2$ using the master congruence above (though in the case $c = 1$, we must exclude the top value $r = 2p - 1$ because of a technicality).

\section{Improving the Bergdall-Levin-Liu bound}\label{Improving the BLL bound}

Throughout this section, let $c = \lfloor r/p \rfloor$ so that $cp \leq r \leq (c + 1)p - 1$. We begin by showing that if $v_p(\cL) < r/2 - n$ for an $n$ satisfying $r/2 + b + 1 \leq n \leq r$ and $v_p([n]_{b + 1}) = 0$, then the sub-quotient $F_{2(r - (n - b - 1)), 2(r - (n - b - 1)) + 1}$ vanishes. This is what we called the good method in the introduction.

\begin{theorem}\label{The good method}
    Fix $cp \leq r \leq (c + 1)p - 1$ with $1 \leq c \leq p - 2$. Let $\cL$ be chosen such that $v_p(\cL) < r/2 - n$ for some integer $n$ satisfying $r/2 + b + 1 \leq n \leq r$ and $v_p([n]_{b + 1}) = 0$. Then, $F_{2(r - (n - b - 1)), 2(r - (n - b - 1)) + 1} = 0$.
\end{theorem}
\begin{proof}
    Let us begin by simplifying the congruence proved in Proposition~\ref{Final congruence} for those $n$ for which $v_p([n]_{b + 1}) = 0$. The congruence implies that
    \begin{eqnarray*}
        0 & \equiv & \sum_{a = 1}^{n - bp}\sum_{j = \lceil r/2 \rceil}^{n - 1}{n \choose j}{\epsilon \choose a}p^{x + n - j}\frac{(-1)^{- a - j + b - 1}}{a}{(b + 1)p \choose n + 1}(n + 1)\cL b!{n - j \brace b}(z - a)^j\mathbbm{1}_{a + p\Zp} \\
        && + \sum_{j = \lceil r/2 \rceil - 1}^{n - 1}{n \choose j}p^{x + n - j}(-1)^{n - j}\cL *_j z^j\mathbbm{1}_{p\Zp} \mod \pi \latticeL{k}.
    \end{eqnarray*}
        The valuation of the coefficient of $(z - a)^j\mathbbm{1}_{a + p\Zp}$ in the first line is $\geq x + n - j + 1 + v_p(\cL)$. Since $v_p([n]_{b + 1}) = 0$, this bound is equal to $r/2 - j + 1$. Therefore, we may use \cite[Lemma 6.4]{CG23} to eliminate the first line on the RHS of the congruence above. 
        
        Next, the valuation of the coefficient of $z^j\mathbbm{1}_{p\Zp}$ in the second line is $\geq x + n - j + v_p(\cL) + v_p(*_j)$. If $j \geq n - b$, then using the mod $p$ congruence for $*_j$ proved in Proposition~\ref{Final congruence}, we see that $v_p(*_j) \geq 1$. So the $j \geq n - b \>\> (\geq r/2 + 1)$ terms in the second line on the RHS of the congruence above can again be ignored and we have 
        \begin{eqnarray}\label{Inductive step ugly argument}
            0 \equiv \sum_{j = \lceil r/2 \rceil - 1}^{n - b - 1}{n \choose j}p^{x + n - j}(-1)^{n - j}\cL *_j z^j\mathbbm{1}_{p\Zp} \mod \pi \latticeL{k}.
        \end{eqnarray}
        
        Since $v_p([n]_{b + 1}) = 0$, we see that $p \nmid {n \choose n - b - 1}$ (note $p \nmid (b + 1)!$ if $b \leq p - 2$). Also $p \nmid *_{n - b - 1}$ by the same proposition. This shows that the $j = n - b - 1  \>\>(\geq r/2)$ term is a generator of $F_{2(r - n + b + 1), 2(r - n + b + 1) + 1}$. Thus, we may use \eqref{Inductive step ugly argument} to kill $F_{2(r - n + b + 1), 2(r - n + b + 1) + 1}$ once we check that the lower degree terms $\lceil r/2 \rceil \leq j \leq n - b - 2$ are integral and that the $j = \lceil r/2 \rceil - 1$ term vanishes. Indeed, the valuation of the coefficient of $z^j\mathbbm{1}_{p\Zp}$ is $\geq r/2 - j$ since $*_j$ is integral. So we use \cite[Lemma 6.4]{CG23} for the $\lceil r/2 \rceil \leq j \leq n - b - 2$ terms and \cite[Lemma 6.3]{CG23} for the $j = \lceil r/2 \rceil - 1$ term to conclude that they are integral.
\end{proof}

Next, we specialize to $cp + c + 2 \leq r \leq (c + 1)p - 1$ for $1 \leq c \leq p - 3$. Using the argument given in Section~\ref{Killing shallow sub-quotients}, we see that all the shallow sub-quotients $F_{0, 1}, F_{2, 3}, \ldots, F_{2(c - 1), 2(c - 1) + 1}$ vanish.

Next, we claim that for $c = 1, 2$, the sub-quotient $F_{2c, 2c + 1}$ is the only one contributing towards $\br{\latticeL{k}}$. So we must eliminate all the deeper sub-quotients. We may separate the deeper sub-quotients into two sets. The first set consists of those generated by $p^{r/2 - (n - b - 1)}z^{n - b - 1}\mathbbm{1}_{p\Zp}$ for some $n$ satisfying $r/2 + b + 1 \leq n \leq r$ for which $v_p([n]_{b + 1}) = 0$ whereas the second set consists of those for which $v_p([n]_{b + 1}) = 1$. The first ones are eliminated using Theorem~\ref{The good method}. This means that the sub-quotients which are possibly left are those generated by $p^{r/2 - j}z^{j}\mathbbm{1}_{p\Zp}$ where $j$ runs through those values in the list below which are greater than or equal to $\lceil r/2 \rceil$:
    \begin{center}
        \begin{tabular}{ c c c c c}
            $cp - c - 1$, & $cp - c$, & $\ldots$, &  $cp - 2,$ & $cp - 1$,\\ 
            $(c - 1)p - c$, & $(c - 1)p - c + 1$, & $\ldots$, & $(c - 1)p - 1$, & \\
            $\vdots$ & $\vdots$ & $\ddots$ & & \\
            $p - 2$ & $p - 1.$ & & &
        \end{tabular}
    \end{center}
Indeed, this list consists of those values of $n - b - 1$ for which $n$ satisfies $v_p([n]_{b + 1}) = 1$. Note that the values of $n - b - 1$ appearing in the first column can also be obtained from $n$ satisfying $v_p([n]_{b + 1}) = 0$. Indeed, if $n = dp - 1$ for any $1 \leq d \leq c$, then $n - b - 1 = dp - 1 - (d - 1) - 1 = dp - d - 1$. Therefore we may ignore the values of $n - b - 1$ appearing in the first column.

Next, we eliminate the sub-quotients corresponding to all the remaining entries except the last one in each row. This is what we call the bad method. Note that this argument is not needed for $c = 1$ since there are only two columns and only one row in the diagram above. For $c = 2$, this argument only applies to $n - b - 1 = 2p - 2$ since the diagram becomes
    \begin{center}
        \begin{tabular}{ c c c }
            $2p - 3$, & $2p - 2$, & $2p - 1$,\\ 
            $p - 2$ & $p - 1.$ & 
        \end{tabular}
    \end{center}
The following theorem treats this case.
\begin{theorem}\label{The bad method}
    Let $2p + 4 \leq r \leq 3p - 1$. Let $n = 2p + 1$ and $v_p(\cL) < r/2 - n$. Then, $F_{2(r - 2p + 2), 2(r - 2p + 2) + 1} = 0$.
\end{theorem}
\begin{proof}
    Proposition~\ref{Final congruence} implies that
            \begin{eqnarray*}
                0 & \equiv & \sum_{a = 1}^{n - bp}\sum_{j = \lceil r/2 \rceil}^{n - 1}{n \choose j}{\epsilon \choose a}p^{x + n - j}\frac{(-1)^{- a - j + b - 1}}{a}{(b + 1)p \choose n + 1}(n + 1)\cL b!{n - j \brace b}(z - a)^j\mathbbm{1}_{a + p\Zp} \\
                && + \sum_{j = \lceil r/2 \rceil - 1}^{n - 1}{n \choose j}p^{x + n - j}(-1)^{n - j}\cL *_j z^j\mathbbm{1}_{p\Zp} \mod \pi \latticeL{k}.
            \end{eqnarray*}
        We claim that the first line on the RHS of the congruence above can be ignored. The valuation of the coefficient of $(z - a)^j\mathbbm{1}_{a + p\Zp}$ is $v_p({n \choose j}) + x + n - j + 1 + v_p(\cL) + v_p({n - j \brace b}) = r/2 - j + v_p({n \choose j}) + v_p({n - j \brace b})$. The 
        $j > n - b$ terms can be ignored since ${n - j \brace b} = 0$. The $j = n - b, n - b - 1$ terms can be ignored using \cite[Lemma $6.4$]{CG23} with Lucas since ${n \choose j} \equiv 0 \mod p$.
        The rest of the terms are integral by the same lemma. So when we project modulo powers of $z$ supported on deeper sub-quotients, we may ignore the first line on the RHS.

        This leaves only the second line on the RHS. The valuation of the coefficient of $z^j\mathbbm{1}_{p \Zp}$ there is $r/2 - j - 1 + v_p({n \choose j}) + v_p(*_j)$. Therefore, we may ignore the $j \geq n - b + 1$ terms using the mod $p^2$ values of $*_j$ given in Proposition~\ref{Final congruence}. To eliminate the $j = n - b$ term, we see that both the binomial coefficient and the Stirling number each contribute at least one power of $p$.
        So we are left with
        \begin{eqnarray*}
            0 & \equiv & {n \choose n - b - 1} p^{x + b + 1}(-1)^{b}\cL(b + 1)!z^{n - b - 1}\mathbbm{1}_{p\Zp} \\
            && \qquad + \sum_{j = \lceil r/2 \rceil - 1}^{n - b - 2} {n \choose j}p^{x + n - j}(-1)^{n - j} \cL *_j z^j\mathbbm{1}_{p\Zp} \mod \pi \latticeL{k},
        \end{eqnarray*}
        where we have substituted for the value of $*_j$ mod $p$ obtained in Proposition~\ref{Final congruence}.
        Note that the coefficient of the first term on the RHS has the correct valuation and is therefore a generator of $F_{2(r - (n - b - 1)), 2(r - (n - b - 1)) + 1}$ which we aim to eliminate. We may do so by going modulo deeper sub-quotients once we show that every term in the sum on the RHS is integral and that the $j = \lceil r/2 \rceil - 1$ term vanishes.
        Since $*_j \equiv -(b + 1)!{n - j \brace b + 1} \mod p$ by Proposition~\ref{Final congruence}, it is enough to check that $r/2 - j - 1 + v_p({n \choose j}) + v_p({n - j \brace b + 1}) \geq r/2 - j$ for all $p + 1 \leq \lceil r/2 \rceil - 1 \leq j \leq n - b - 2 = 2p - 3$. The binomial coefficient vanishes mod $p$ by Lucas for every $j$ except possibly $p + 1$.
        Moreover, $j = p + 1$ can only occur for $r = 2p + 4$. Then, the Stirling number comes to our rescue since ${n - j \brace b} \equiv {1 \brace 2} = 0 \mod p$ using Lemma~\ref{Lucas for Stirling}. Thus, we have shown that every term in the sum above is integral. Note that as before the $j = \lceil r/2 \rceil - 1$ term vanishes by \cite[Lemma $6.3$]{CG23}. This concludes the proof of this theorem.
\end{proof}

Finally, the only deeper sub-quotient that is left in the case $cp + c + 2 \leq r \leq (c + 1)p - 1, c = 1, 2$ is $F_{2(r - (n - b - 1)), 2(r - (n - b - 1)) + 1}$ with $n - b - 1 = cp - 1$, which is the last entry in the first row in the initial diagram (when $c = 2$, the values in the second row are strictly less than $\lceil r/2 \rceil$). To eliminate this sub-quotient, we need the ugly method.

\begin{theorem}\label{The ugly method}
    Let $cp + c + 2 \leq r \leq (c + 1)p - 1$ with $c = 1, 2$ and $v_p(\cL) < r/2 - (cp + c + 1)$. Then, $F_{2(r - cp + 1), 2(r - cp + 1) + 1} = 0$.
\end{theorem}
\begin{proof}
    Proposition~\ref{Final congruence} applied to $n = cp + c$ gives us
        \begin{eqnarray*}
        0 & \equiv & \sum_{a = 1}^{n - bp}\sum_{j = \lceil r/2 \rceil}^{n - 1}{n \choose j}{\epsilon \choose a}p^{x + n - j}\frac{(-1)^{- a - j + b - 1}}{a}{(b + 1)p \choose n + 1}(n + 1)\cL b!{n - j \brace b}(z - a)^j\mathbbm{1}_{a + p\Zp} \\
        && + \sum_{j = \lceil r/2 \rceil - 1}^{n - 1}{n \choose j}p^{x + n - j}(-1)^{n - j}\cL *_j z^j\mathbbm{1}_{p\Zp} \mod \pi \latticeL{k}.
    \end{eqnarray*}
    The valuation of the coefficient of $(z - a)^j\mathbbm{1}_{a + p\Zp}$ in the first line is $v_p({n \choose j}) + x + n - j + 1 + v_p(\cL) + v_p({n - j \brace b}) = r/2 - j + v_p({n \choose j}) + v_p({n - j \brace b})$. So the corresponding function is integral.
    The $j \geq n - b + 1$ terms die as ${n - j \brace b} = 0$. The $j = n - b - 1$ term dies since $v_p({n \choose b + 1}) = v_p([n]_{b + 1}) = 1$. This leaves only the $j = n - b$ term and the $\lceil r/2 \rceil \leq j \leq n - b - 2$ terms in the first line.
    
    In the second line, the valuation of the coefficient of $z^j\mathbbm{1}_{p\Zp}$ is $r/2 - j - 1 + v_p({n \choose j}) + v_p(*_j)$. The $j \geq n - b + 1$ terms vanish since $*_j \equiv 0 \mod p^2$ by Proposition~\ref{Final congruence}. The $j = n - b$ term is integral since $*_j \equiv 0 \mod p$. And, the $j = n - b - 1$ term is a generator of $F_{2(r - (n - b - 1)), 2(r - (n - b - 1)) + 1}$ since $v_p({n \choose n - b - 1}) = 1$ and $*_j \equiv -(b + 1)! \mod p$. For the $\lceil r/2 \rceil - 1 \leq j \leq n - b - 2$ terms, one can prove that $p \mid {cp + c \choose j} *_j$ making them integral. For $c = 1$, this is true since $p \mid {p + 1 \choose j}$. For $c = 2$, the binomial coefficient contributes a power of $p$ for every $(j, r) \neq (p + 1, 2p + 4), (p + 2, 2p + 4), (p + 2, 2p + 5), (p + 2, 2p + 6)$.
    For the first case,
    \[
        *_j \equiv -{n - j \brace b + 1}(b + 1)! = -{p + 1 \brace 3}3! \equiv -{2 \brace 3}3! = 0 \mod p
    \]
    by Lemma~\ref{Lucas for Stirling}.
    For the remaining three cases,
    \[
        *_j \equiv -{n - j \brace b + 1}(b + 1)! = -{p \brace 3}3! \equiv -{1 \brace 3}3! = 0 \mod p.
    \]
    So we may mod out powers of $z$ supported on sub-quotients deeper than $F_{2(r - (n - b - 1)), 2(r - (n - b - 1)) + 1}$ and get
    \begin{eqnarray*}
        0 & \equiv & {n \choose n - b - 1}p^{x + b + 1}(-1)^{b}\cL(b + 1)!z^{n - b - 1}\mathbbm{1}_{p\Zp} + f,
    \end{eqnarray*}
    where $f$ is an integral linear combination of terms of the form $p^{r/2 - n + b}(z - i)^{n - b}\mathbbm{1}_{i + p\Zp} = p^{r/2 - cp}(z - i)^{cp}\mathbbm{1}_{i + p\Zp}$ for $0 \leq i \leq c$. We may conclude that $F_{2(r - (n - b - 1)), 2(r - (n - b - 1)) + 1} = 0$ once we eliminate $f$ from this congruence, which we do now. Reapplying Proposition~\ref{Final congruence} to $n = cp + c + 1$ (note that $v_p([n]_{b + 1}) = 0$), we conclude (see \eqref{Inductive step ugly argument})
    \[
        0 \equiv {cp + c + 1 \choose cp}p^{x + c + 1}(-1)^{c}\cL (c + 1)!z^{cp}\mathbbm{1}_{p\Zp} + {cp + c + 1 \choose cp - 1}p^{x + c + 2}(-1)^{c + 2}\cL *_{cp - 1}z^{cp - 1}\mathbbm{1}_{p\Zp}
    \]
    modulo terms supported on deeper sub-quotients. The final idea is that since $p \mid {cp + c + 1 \choose cp - 1}$ (by Lucas noting that $c \leq 2 < p - 2$), the second term in the congruence above can be ignored using \cite[Lemma 6.3, Lemma 6.4]{CG23}, and therefore $p^{r/2 - cp}z^{cp}\mathbbm{1}_{p\Zp}$ is congruent to terms supported on sub-quotients deeper than $F_{2(r - cp + 1), 2(r - cp + 1) + 1}$ modulo $\pi\latticeL{k}$. This means that we may eliminate $f$ from the previous congruence. Therefore, we have finally proved that $F_{2(r - cp + 1), 2(r - cp + 1) + 1} = 0$.
\end{proof}

\begin{theorem}\label{Main theorem}
    Let $p \geq 5$ and $p + 3 \leq r \leq 2p - 2$ or $2p + 4 \leq r \leq 3p - 1$. Then, for any $\cL$ with $v_p(\cL) < -r/2$, we have $\overline{V}_{k, \cL} \simeq \ind \omega_2^{r + 1}$.
\end{theorem}
\begin{proof}
    First, assume $p + 3 \leq r \leq 2p - 1$. Using Theorem~\ref{The good method}, the discussion after it, and Theorem~\ref{The ugly method}, the only sub-quotient left standing is $F_{2, 3}$. Moreover, $\IZind d^{r - 2} \otimes \det \simeq \IZind ad^{r - 1}$ surjects onto $F_{2, 3}$. Since $r - 2 \not \equiv 1, p - 2 \mod p - 1$ for any $p + 3 \leq r \leq 2p - 2$, the reducible possibility in the mod $p$ local Langlands correspondence cannot occur. This is the reason we drop $r = 2p - 1$.

    Next, assume $2p + 4 \leq r \leq 3p - 1$. Using Theorems~\ref{The good method} (and the discussion after), \ref{The bad method}, and \ref{The ugly method} as before, we see that $F_{4, 5}$ is the only sub-quotient left standing. Moreover, $\IZind d^{r - 4} \otimes \det^{2} \simeq \IZind a^2d^{r - 2}$ surjects onto $F_{4, 5}$. Since $r - 4 \not\equiv 1, p - 2 \mod p - 1$ for any $2p + 4 \leq r \leq 3p - 1$, the reducible possibility in the mod $p$ local Langlands correspondence cannot occur.

    In both cases, one may check using \cite[Theorem 2.2]{Chi25} that $\br{V}_{k, \cL} \simeq \ind \omega_2^{r + 1}$.
\end{proof}

\begin{remark}
    As mentioned in the introduction, this recovers the result of \cite{BLL23} for $p + 3 \leq r \leq 2p - 2$ and improves the bound in \cite{BLL23} for $2p + 4 \leq r \leq 3p - 1$.
\end{remark}
\vspace{0.2cm}
\begin{center}
    \textbf{Acknowledgements}
\end{center}
    The first author was supported by the National Research Foundation of Korea (NRF) grant funded by the Korea government (MSIT) (No. RS-2025-02262988 and No. RS-2025-00517685). The second author acknowledges the support of project 1303/9/2025-R\&D-II-DAE/TIFR-17312.

\bibliographystyle{alpha}
\bibliography{References.bib}

\vspace{1cm}
\noindent
\textbf{Anand Chitrao} \\
Department of Mathematical Sciences, UNIST, Ulsan - 44919, Republic of Korea \\
Email: \texttt{anand@unist.ac.kr}

\medskip \noindent\textbf{Eknath Ghate} \\
School of Mathematics, Tata Institute of Fundamental Research, Homi Bhabha Road, Mumbai-400005, India \\
Email: \texttt{eghate@math.tifr.res.in}

\end{document}